\documentclass[12pt]{amsart}

\usepackage{amsmath}
\usepackage{amssymb}  
\usepackage{latexsym} 
\usepackage{comment}
\usepackage{url}
\usepackage{url,amssymb,amsmath,amsthm,amsfonts,mathrsfs}
\usepackage[usenames,dvipsnames]{color}
\usepackage[pagebackref = true, colorlinks = true, linkcolor = blue, citecolor = Green]{hyperref}
\usepackage[alphabetic,lite,nobysame]{amsrefs}
\usepackage{enumitem}
\usepackage{amscd}   
\usepackage[all, cmtip]{xy} 
\usepackage{tikz-cd}

\usepackage{xfrac}

\DeclareFontFamily{U}{wncy}{}
\DeclareSymbolFont{mcy}{U}{wncy}{m}{n}
\DeclareMathSymbol{\Sh}{\mathord}{mcy}{"58} 
\usepackage{color} 

\newcommand{\defi}[1]{\textsf{#1}} 
\DeclareFontEncoding{OT2}{}{} 
\newcommand{\textcyr}[1]{%
 {\fontencoding{OT2}\fontfamily{wncyr}\fontseries{m}\fontshape{n}\selectfont #1}}
\newcommand{\Sha}{{\mbox{\textcyr{Sh}}}}

\def\act#1#2%
  {\mathop{}%
   \mathopen{\vphantom{#2}}^{#1}%
   \kern-3\scriptspace%
   #2}

\newcommand{\Z}{{\mathbb Z}}
\newcommand{\Q}{{\mathbb Q}}
\newcommand{\R}{{\mathbb R}}
\newcommand{\C}{{\mathbb C}}
\newcommand{\F}{{\mathbb F}}
\newcommand{\A}{{\mathbb A}}
\newcommand{\PP}{{\mathbb P}}
\newcommand{\G}{{\mathbb G}}

\newcommand{\Kbar}{{\overline{K}}}

\newcommand{\kbar}{{\overline{k}}}

\newcommand{\Lbar}{{\overline{L}}}

\newcommand{\Xbar}{{\overline{X}}}
\newcommand{\Ybar}{{\overline{Y}}}

\newcommand{\calA}{{\mathcal A}}

\newcommand{\calO}{{\mathcal O}}

\DeclareMathOperator{\Hom}{Hom}

\DeclareMathOperator{\divv}{div}
\DeclareMathOperator{\ord}{ord}

\DeclareMathOperator{\Gal}{Gal}
\DeclareMathOperator{\Cor}{Cor}

\DeclareMathOperator{\Br}{Br}
\DeclareMathOperator{\Sym}{Sym}
\DeclareMathOperator{\Div}{Div}

\DeclareMathOperator{\Pic}{Pic}

\DeclareMathOperator{\Alb}{Alb}

\DeclareMathOperator{\Sel}{Sel}
\DeclareMathOperator{\HH}{H}

\DeclareMathOperator{\Spec}{Spec}

\DeclareMathOperator{\inv}{inv}

\newcommand{\res}{\operatorname{res}}

\newtheorem{Theorem}{Theorem}[section]
\newtheorem{Lemma}[Theorem]{Lemma}
\newtheorem{Proposition}[Theorem]{Proposition}
\newtheorem{Corollary}[Theorem]{Corollary}

\theoremstyle{definition}
\newtheorem{Definition}[Theorem]{Definition}
\newtheorem{Example}[Theorem]{Example}
\newtheorem{Remark}[Theorem]{Remark}
\newtheorem{Algorithm}[Theorem]{Algorithm}

\numberwithin{equation}{section}

\begin{document}

\title{Explicit Brauer-Manin obstructions on plane quartics}

\author{Nils Bruin}
\address{Department of Mathematics, Simon Fraser University, Burnaby BC V5A 1S6, Canada}
\email{nbruin@sfu.ca}
\urladdr{https://www.cecm.sfu.ca/~nbruin}
\author{Brendan Creutz}
\address{School of Mathematics and Statistics, University of Canterbury, Private Bag 4800, Christchurch 8140, New Zealand}
\email{brendan.creutz@canterbury.ac.nz}
\urladdr{http://www.math.canterbury.ac.nz/\~{}bcreutz}

\maketitle

\begin{abstract}
We describe a method to show a plane quartic over a number field has no rational points. The method can be adapted to show that a curve does not have divisors of degree 1 or 2 and can be generalized to arbitrary smooth projective curves. Our approach significantly improves on the applicability over previous 2-cover descent methods by not requiring the computation of the full $S$-unit group of the \'etale algebras involved. We illustrate the practicality with several examples, including examples where we determine plane quartics to be of index 2 or 4 when the maximum local index is strictly smaller.
\end{abstract}

\section{Introduction}

We describe a practical method to show that a plane quartic curve $X$ over a number field $k$ has no rational points, even if it has points everywhere locally. The strategy can also be used to prove that $X$ has no $k$-rational divisors of degree $1$ or $2$. 

More specifically, we compute obstructions that arise from $2$-cover descent. The main innovation over previous work \cite{BruinStoll,BPS, CreutzANTS,CreutzGenJacs} is that we do not need to fully determine $S$-unit groups of number fields to obtain unconditional results: we are already able to obtain nontrivial results by knowing some explicit $S$-units. Our algorithm generalizes the one developed in \cite{CreutzSrivastava} for hyperelliptic curves, and we present a much simplified proof of correctness. Like \cite{BPS}, our approach can be readily generalized to apply to arbitrary smooth projective curves, although computational costs rise exponentially with the genus for non-hyperelliptic curves.

In Section~\ref{sec:etale} we consider an \'etale algebra $L$ over a number field $k$ and a finite set of primes $S$. In Theorem~\ref{thm:pairing} we use Hilbert symbols to establish a pairing which imposes constraints on the image of the localization map
\[
\left(\frac{L^\times}{k^\times L^{\times 2}}\right)_S \to \prod_{v \in S} \frac{L_v^\times}{k_v^\times L_v^{\times 2}}\,.
\]
coming from $S$-units of square norm.

In Section~\ref{sec:obstruction_quartics} we explain how to use this to obtain an obstruction to rational points and/or divisors. We illustrate the argument here for rational points on a plane quartic $X$ over a number field $k$. The $28$ bitangents give rise to an $L$ and a map $C_f\colon X(k)\to L^\times / k^\times L^{\times 2}$ as well as local versions. For a suitable finite set of primes $S$, the map factors through the part unramified outside $S$,
\[C_f(X(k)) \to (L^\times/k^\times L^{\times 2})_S \to \prod_{v\in S} C_f(X(k_v)).\]
Given suitable $S$-units of square norm we may be able to show that the image of the localization map does not intersect $\prod_{v\in S} C_f(X(k_v))$, in which case we can conclude that $C_f(X(k))$ is empty and, hence, that $X(k)$ is empty. This approach contrasts that of \cite{BPS} where one must compute generators for all of $(L^\times/k^\times L^{\times 2})_S$ to obtain the image in the localization. Both approaches require computing the local images $C_f(X(k_v))$. In Remark~\ref{rmk:unram} we describe several ways to reduce the number of primes where this is required.

In Section~\ref{sec:examples} we use our Magma~\cite{magma} implementation \cite{Github} 
on several plane quartic curves over $\Q$ occurring in the literature to prove they have no rational points. In Section~\ref{sec:Index} we describe how the methods can be used to determine the index of a plane quartic and give examples of curves of index $2$ or $4$ where the maximum local index is strictly smaller. Finally, in Section~\ref{sec:Descent} we describe how our obstruction can be interpreted in terms of finite abelian descent and Brauer-Manin obstructions. In particular, Corollary~\ref{Cor:Br2} shows that our obstruction is equivalent to these obstructions, for suitably large $S$.\\

\noindent{\bf Acknowledgements:}
The first author acknowledges support from Canada's NSERC, RGPIN-2024-04233.
The second author was supported in part by the Marsden Fund administered by Royal Society Te Ap\=arangi.

\section{\'Etale algebras}\label{sec:etale}

Let $L/k$ be an \'etale algebra over a number field $k$. Equivalently, $L \simeq \prod_{i =1}^m K_i$ is a product of finite field extensions $K_i/k$. Let $v$ be a prime of $k$ with completion $k_v$. Then $L_v := L \otimes_k k_v$ decomposes as the product $L_v\simeq \prod_{i = 1}^m \prod_{w \mid v} K_{i,w}$ of the completions of the $K_i$ at the primes of $K_i$ above $v$. We write $\kbar$ for an algebraic closure of $k$ and $\Lbar = L \otimes \kbar$.

\begin{Definition}
A class in $L_v^\times/L_v^{\times 2}$ is \defi{unramified} if it is represented by some $\ell_v = (\ell_{i,w}) \in L_v \simeq \prod K_{i,w}$ such that all of the extensions $K_{i,w}(\sqrt{\ell_{i,w}})/K_{i,w}$ are unramified. A class in $L_v^\times/k_v^\times L_v^{\times 2}$ is \defi{unramified} if it can be represented by a class in $L_v^\times/L_v^{\times 2}$ that is unramified. For a finite set of primes $S$ of $k$, we write $(L^\times/k^\times L^{\times 2})_S$ and $(L^\times/L^{\times 2})_S$ for the subgroups of classes whose images in $L_v/k_v^\times L_v^{\times 2}$ (resp. in $L_v^\times/L_v^{\times 2}$) are unramified at all $v \not\in S$.
\end{Definition}
%

The norm map $N_{L/k}\colon L \to k$ induces a homomorphism $L^\times/L^{\times 2} \to k^\times/k^{\times 2}$. We write $(L^\times/L^{\times 2})_{N = 1}$ to denote the kernel of this homomorphism and $(L^\times/L^{\times 2})_{S,N = 1}$ for its intersection with $(L^\times/L^{\times2})_S$.

On each factor of $L_v\simeq \prod_{i = 1}^m \prod_{w \mid v} K_{i,w}$, the Hilbert symbol defines a nondegenerate bilinear pairing on the finite dimensional $\F_2$-vector space $K_{i,w}^\times/K_{i,w}^{\times 2}$ taking values in $\{0,\sfrac{1\!}{2}\} \subset \Q/\Z$. Summing these over the factors yields a bilinear pairing of finite dimensional $\F_2$-vector spaces:
\begin{equation}\label{eq:Hilb}
	(\,,\,)_v\colon \frac{L_v^\times}{L_v^{\times 2}} \times \frac{L_v^\times}{L_v^{\times 2}} \to \Q/\Z\,.
\end{equation}

\begin{Theorem}\label{thm:pairing}
	Let $L/k$ be an \'etale algebra over a number field $k$ and let $S$ be a finite set of primes of $k$. 
	\begin{enumerate}
		\item\label{pairing1} The pairings~\eqref{eq:Hilb} induce a bilinear pairing of finite dimensional $\F_2$-vector spaces
		\[
			\langle\,,\rangle_S \quad \colon \quad \prod_{v \in S} \frac{L_v^\times}{k_v^\times L_v^{\times 2}} \times \left(\frac{L^{\times}}{L^{\times 2}}\right)_{S,N = 1} \to\; \Q/\Z
		\]
		by
		\[
			\langle (\ell_v)_{v\in S} , m \rangle_S= \sum_{v\in S} (\ell_v,\res_v(m))_v \in \Q/\Z\,.
		\]
		\item\label{pairing2} The left kernel of the pairing $\langle\,,\rangle_{S}$ is equal to the image of the restriction map 
		\[
			\left(\frac{L^\times}{k^\times L^{\times 2}}\right)_S \to \prod_{v \in S} \frac{L_v^\times}{k_v^\times L_v^{\times 2}}\,.
		\]
	\end{enumerate}
\end{Theorem}

\begin{proof}
	We first prove~\eqref{pairing1}. Let $v$ be a prime of $k$. It suffices to show that the pairing~\eqref{eq:Hilb} induces a bilinear pairing 
	\begin{equation}\label{Hilb2}
		\langle\,,\,\rangle_v \colon L_v^\times/k_v^\times L_v^{\times 2} \times (L_v^\times/L_v^{\times 2})_{N = 1} \to \Q/\Z\,.
	\end{equation}
	Summing these over the primes in $S$ gives the pairing in the statement.
	
	Let $a \in k_v^\times \subset L_v^\times$ and let $\ell = (\ell_{i,w}) \in L_v^\times \simeq \prod_i\prod_{w\mid v}K_{i,w}^\times$ be an element with $N_{L_v/k_v}(\ell) \in k_v^{\times 2}$. We must show $(a,\ell)_v = 0$. 
	
	Since $a \in k_v^\times$, the Hilbert symbols on $K_{i,w}$ and $k_v$ are related by
	\begin{equation}\label{eq:Cor}
		(a,\ell_{i,w})_{K_{i,w}} = (a,N_{K_{i,w}/k_v}(\ell_{i,w}))_{k_v}\,.
	\end{equation}
	This follows from \cite[XI \S2 Prop.1(ii)]{Serre} and \cite[p. 205, Exercise]{Serre}.
	
	The pairing $(a,\ell)_v$ is the sum of the pairings~\eqref{eq:Cor}. Since the Hilbert symbol on $k_v$ is bilinear and $N_{L_v/k_v}(\ell) \in k_v^{\times 2}$, we have
	\[
		(a,\ell)_v = \sum_{i}\sum_{w\mid v} (a,N_{K_{i,w}/k_v}(\ell_{i,w}))_{k_v} = (a,N_{L_v/k_v}(\ell))_{k_v} = 0\,.
	\]
		Now we prove~\eqref{pairing2}. Let $k_S/k$ be the maximal extension unramified outside $S$. We consider Galois cohomology of the self-dual $\Gal(k_S/k)$-modules $\mu_2(\kbar)$ and $\mu_2(\Lbar)$. Their respective Kummer sequences together with Hilbert's Theorem 90 give isomorphisms $\HH^1(k_S/k,\mu_2) \simeq (k^\times/k^{\times 2})_S$, and $\HH^1(k_S/k,\mu_2(\Lbar)) \simeq (L^\times/L^{\times 2})_S$. Under this identification the Poitou-Tate exact sequences for these $\Gal(k_S/k)$-modules \cite[Theorem 4.10]{Milne} give rise to a commutative diagram of linear maps of $\F_2$-vector spaces with exact rows:
		\begin{equation}\label{eq:Hilb1}
		\begin{tikzcd}
			(k^\times/k^{\times 2})_S \arrow[r] \arrow[d,"i_*"]& \prod_{v\in S} k_v^\times/k_v^{\times 2} \arrow[r] \arrow[d,"i_*"]& \left(k^\times/k^{\times 2}\right)_S^* \arrow[r]\arrow[d,"N_{L/k}^*"]& 1\\	
			\left(L^\times/L^{\times 2}\right)_S \arrow[r]& \prod_{v\in S} L_v^\times/L_v^{\times 2} \arrow[r]& \left(L^\times/L^{\times 2}\right)_S^*\,,
		\end{tikzcd}
		\end{equation}
		where the ${}^*$ denotes the dual vector space, the vertical maps are induced by either the inclusion $i \colon \mu_2(\kbar) \to \mu_2(\Lbar)$ or the norm $N_{L/k}\colon\mu_2(\Lbar) \to \mu_2$, and the second map in the bottom row is given by
		\[
			(\ell_v)_{v \in S} \mapsto \left[ m \mapsto \sum_{v\in S} (\ell_v,m)_v \right]\,.
		\]
		To see that the map to $\left(k^\times/k^{\times 2}\right)_S^*$ is surjective note that the Poitou-Tate sequence in the first row continues
		\[
			\left(k^\times/k^{\times 2}\right)_S^* \to \HH^2(k_S,\mu_2) \to \prod_{v \in S}\HH^1(k_v,\mu_2)\,,
		\]
		and that the second map here is injective by global class field theory. For the reader less familiar with duality theorems in Galois cohomology we note that exactness of~\eqref{eq:Hilb1} can be deduced directly from the well-known Hilbert reciprocity law.
		
		The restriction map and the pairing $\langle\,,\,\rangle_S$ give maps 
		\begin{equation}\label{eq:Hilb2}
			\left(\frac{L^\times}{k^\times L^{\times 2}}\right)_S \to\prod_{v \in S} \frac{L_v^\times}{k_v^\times L_v^{\times 2}} \to \left(\frac{L^\times}{L^{\times 2}}\right)^*_{S,N=1}\,.
		\end{equation}
		The bottom row of~\eqref{eq:Hilb1} maps surjectively onto~\eqref{eq:Hilb2}, from which it follows that~\eqref{eq:Hilb2} is an exact sequence. This exactness is equivalent to the claim that the left kernel of $\langle \,,\, \rangle_{S}$ is the image of the restriction map.
\end{proof}

\subsection{Finding elements in $\left(L^\times/L^{\times 2}\right)_{S}$}\label{sec:findingelements}

The applications we have in mind concern the image of the first map in~\eqref{eq:Hilb2}. Because of the exactness we can compute it either as an image or as a kernel. As we explain below, in either case we end up considering $S$-units in $L^\times$, either in order to represent preimages or to represent classes in $\left(L^\times/L^{\times 2}\right)_{S,N = 1}$ that constrain the image through the pairing $\langle\,,\,\rangle_S$.

We assume that $S$ is a finite set of primes of $k$ including all archimedean primes and all primes above $2$. Then there is an exact sequence
\begin{equation}\label{eq:ClassGroup}
	0 \to \left(\calO_{L,S}^\times/\calO_{L,S}^{\times 2}\right) \to \left(L^\times/L^{\times 2}\right)_{S} \to \operatorname{Cl}(\calO_{L,S})[2] \to 0\,,
\end{equation}
where $\calO_{L,S} = \prod_{i=1}^m \calO_{K_i,S}$ is the product of the rings of $S$-integers of the direct factors $K_i$ of $L$ and $\operatorname{Cl}(\calO_{L,S}) = \prod_{i = 1}^m \operatorname{Cl}(\calO_{K_i,S})$ is the product of their ideal class groups \cite[Proposition 12.6]{PoonenSchaefer}.

Let $K$ be one of the number field factors of $L$. The index calculus techniques of Buchmann's subexponential class- and unit-group algorithm \cite{Buchmann1990} apply to determining $\calO_{K,S}^\times$. We review them here to explain how partial information can be extracted from the procedure. Let $T$ be a set of primes (a `factor basis'). Assuming $S \subset T$ and setting $U = T \setminus S$, there is an exact sequence
\[
	1 \longrightarrow \calO_{K,S}^\times \longrightarrow \calO^\times_{K,T} \stackrel{\phi}\longrightarrow \bigoplus_{\frak{p}\in U}\Z\frak{p} \stackrel{\pi}\longrightarrow \operatorname{Cl}(\calO_{K,S})\,,
\]
where the exponent vector $\phi(\alpha) = (\ord_\frak{p}(\alpha))_{\frak{p}\in U}$ is called a \emph{relation}. A set of elements $\alpha_i \in \calO_{K,T}^\times$ determines a \emph{relation matrix} $M$ with the $\phi(\alpha_i)$ as rows. The left kernel of $M$ determines a subgroup of  $\calO_{K,S}^\times$, generated by the appropriate power products of the $\alpha_i$. A relation matrix $M$ also determines a subgroup of $\calO_{K,S}^{\times}/\calO_{K,S}^{\times 2}$, given by the left kernel of $M \otimes \F_2$, which can be computed using linear algebra over $\F_2$. Doing this for each direct factor $K_i$ of $L$, we obtain a basis for a subgroup of $\calO_{L,S}^\times/\calO_{L,S}^{\times 2}$. By computing the norms of the $\alpha_i$, we can find those linear combinations which lie in the subgroup $(\calO_{L,S}^\times/\calO_{L,S}^{\times 2})_{N=1} \subset (L^\times/L^{\times 2})_{S,N=1}$.

For determining the $S$-unit group and the class group of $\calO_{K,S}$ one takes $T$ large enough to ensure that $\pi$ is surjective (for example if $T$ contains all primes up to the Minkowski bound, or one of the better bounds known to hold conditionally on GRH), and one generates enough $\alpha_i$ to ensure the relations generate the full unit group. For our application we only need to sufficiently constrain the image of the first map in~\eqref{eq:Hilb2}. We can do so with a relation matrix for which the left kernel of $M\otimes \F_2$ is sufficiently large to yield a subgroup of $(L^\times/L^{\times 2})_{S,N=1}$ that under the pairing $\langle\,,\,\rangle_S$ provides sufficient constraints. We don't need surjectivity under $\pi$ or a proof that our subgroup is of index $1$. See Remark~\ref{rem:notall} for an explicit example where we carry out the method using a subgroup of large index.


\section{An obstruction to rational points on plane quartics}
\label{sec:obstruction_quartics}

Let $k$ be a number field with ring of integers $\calO_k$ and let $X/k$ be a smooth quartic curve in $\PP^2$ defined by
\[X\colon g(x,y,z)=0, \text{ with } g(x,y,z) \in \calO_k[x,y,z]\]
 homogeneous of degree $4$. For an extension $K/k$ we denote the base change of $X$ to $K$ by $X_K$ and we write $\Xbar$ for the base change of $X$ to $\kbar$.

We use $\Pic(X)$ to denote the group of $k$-rational divisors on $X$ modulo principal divisors, and use $\Pic_X$ to denote the $k$-scheme representing the Picard functor. We use $\Pic^i(X)$ for the subset of divisor classes of degree $i$ and $\Pic^i_X$ for the degree $i$ component of the Picard scheme. For an extension $K$ of $k$, the rational points $\Pic^i_X(K)$ are \emph{divisor classes} of degree $i$ that are defined over $K$, while $\Pic^i(X_K)$ consists of classes of \emph{divisors} of degree $i$ that are defined over $K$. We note that the natural injective group homomorphism $\Pic(X_K) \to \Pic_X(K)$ need not be surjective. 

We identify $X$ with its image in $\Pic^1_X$. For a subscheme $Y\subset \Pic^i_X$ we introduce a shorthand for the intersection:  \[Y(K)':=Y(K)\cap \Pic^i(X_K).\] In what follows we mainly consider $Y=X$ and $Y=\Pic^i_X$.
%

The $28$ bitangents to $X$ form a $k$-scheme $\Delta = \Spec(L)$, where $L$ is an \'etale $k$-algebra of degree $28$. We write $k(X\times \Delta)$ for the $k$-algebra of rational functions on $X\times \Delta$. Given a bitangent $\delta \in \Delta(\kbar)$, let $\beta_\delta \in \Div(\Xbar)$ denote its contact points, viewed as a divisor of degree $2$ on $\Xbar$. These piece together to give a divisor $\beta \in \Div(X \times \Delta)$. Fix a linear form $ux + vy + wz$ with $u,v,w \in \calO_L$ defining the generic bitangent and a nonzero linear form $h \in k[x,y,z]$. Then $f := (ux+vy+wz)/h \in k(X \times \Delta)^\times$ is a function whose divisor is $\divv(f) = 2\beta - \kappa_X\times \Delta$, where $\kappa_X$ is a canonical divisor on $X$. Algorithms for computing these data from the quartic form $g(x,y,z)$ are described in~\cite[Section 12]{BPS}, as is a formula for computing the discriminant of a plane quartic.

\begin{Lemma}\label{lem:descentmap}
	The function $f \in k(X \times \Delta)^\times$ defines, functorially in the extension $K/k$, homomorphisms 
	\[ C_f \colon \Pic(X_K) \to \frac{(L\otimes K)^\times}{K^\times  (L\otimes K)^{\times 2}}\,,\]
	sending the class of a divisor $D = \sum_{P \in X(\Kbar)} n_PP \ \in \Div(X_K)$ with support disjoint from the zeros and poles of $f$ to the class of $f(D) = \prod_{P}f(P)^{n_P}$ in $(L\otimes K)^\times/K^\times  (L\otimes K)^{\times 2}$.
\end{Lemma}

\begin{proof}
	The rule in the statement defines a homomorphism on the subgroup of $\Div(X_K)$ of divisors with support disjoint from the support of $f$. Using Weil reciprocity one checks that $f(D)$ depends only on the class of $D$ in $\Pic(X_K)$. Since every divisor on $X_K$ is linearly equivalent to one whose support is disjoint from the support of $f$, the homomorphism extends in a unique way to $\Pic(X_K)$. See \cite[6.3.2]{BPS} and \cite[Lemma 4.3]{CreutzGenJacs} for more details.
\end{proof}

\begin{Definition}\label{def:VBS}
Let $i \ge 0$, let $S$ be a finite set of primes of $k$, and let $B \subset (L^\times/L^{\times 2})_{S,N=1}$ be any subgroup.
For $Y\subset \Pic_X^i$ we define
\begin{align*}
	V_{B,S}(Y) &:= \left\{ (\ell_v) \in \prod_{v \in S} \frac{L_v^\times}{k_v^\times L_v^{\times 2}} \;\colon\; 
    \begin{aligned}
        &\forall\,b \in B\,,\, \langle (\ell_v),b\rangle_S = 0, \text{ and} \\
        &\forall\, v\in S\,,\, \ell_v \in C_f(Y(k_v)')
    \end{aligned} 
    \right\}\,.
\end{align*}
\end{Definition}

\begin{Theorem}\label{newthm1}
	With the notation of Definition~\ref{def:VBS}, assume that for all $v \notin S$ the image of $C_f \colon Y(k_v)' \to L_v^\times/k_v^{\times}L_v^{\times 2}$ is unramified. Then $C_f(Y(k)')\subset V_{B,S}(Y)$, and hence $Y(k)'$ is empty if $V_{B,S}(Y)$ is empty.
	Furthermore, if $V_{B,S}(X)=\emptyset$ then $X(k)=\emptyset$. 
\end{Theorem}

\begin{proof}
	By assumption the image of $C_f\colon Y(k)' \to L^\times/k^\times L^{\times 2}$ is unramified outside $S$.
	By~Theorem~\ref{thm:pairing} we have the following commutative diagram of set maps whose bottom row is an exact sequence of $\F_2$-vector spaces:
	
	\[\begin{tikzcd}
		Y(k)' \arrow[d,"C_f"] \arrow[rr] && \displaystyle\prod_{v \in S} Y(k_v)' \arrow[d,"\prod C_f"] \\
		\left(\dfrac{L^\times}{k^\times L^{\times 2}}\right)_S \arrow[rr,"\prod \res_v"] && {\displaystyle\prod_{v \in S} \dfrac{L_v^\times}{k_v^\times L_v^{\times 2}}} \arrow[r] &  \left(\dfrac{L^\times}{L^{\times 2}}\right)^*_{S,N=1}\,.
	\end{tikzcd}
	\]
	The second map in the bottom row is induced by $\langle\,,\,\rangle_S$. It follows that the image of any $x \in Y(k)'$ in $\prod_{v \in S} \frac{L_v^\times}{k_v^\times L_v^{\times 2}}$ must lie in $V_{B,S}(Y)$. Hence, if $V_{B,S}(Y) = \emptyset$, then we must have $Y(k)' = \emptyset$.
	
	Note that points on $X$ are the unique effective representatives in their degree $1$ divisor class, so $X(k)=X(k)'$, giving us the last claim.
\end{proof}

\begin{Algorithm}\label{alg1}~

\noindent\textbf{Input:} $g\in\calO_k[x,y,z]$, the $k$-algebra $L$, a set $S$, $b_1,\ldots,b_n\in L^\times$ generators for a subgroup $B \subset (L^\times/L^{\times 2})_{S,N=1}$, and $Y$ as in Definition~\ref{def:VBS}.

\noindent\textbf{Output:} The set $V_{B,S}(Y)$.

\begin{enumerate}
	\item\label{step1} Compute bases for $V = \prod_{v\in S}L_v^\times/k_v^\times L_v^{\times 2}$ and $W = \prod_{v\in S}(L_v^\times/L_v^{\times 2})_{N=1}$ as well as a function representing the map $L^\times \to \prod_{v \in S}L_v^\times/L_v^{\times 2}$.
	\item\label{step2} Find the matrix representing the pairing $\langle\,,\,\rangle_S \colon V \times W \to \Q/\Z$ of Theorem~\ref{thm:pairing} with respect to these bases.
	\item\label{step3} For each $i = 1,\dots, n$, find the image of $b_i$ in $W$ and its orthogonal complement $V^i \subset V$ with respect to the pairing. 
	\item\label{step4} Let $V_0 = \bigcap_{i = 1}^n V^i$.
	\item\label{step5} For each $v \in S$ compute the image $I_v$ of $C_f \colon Y(k_v)' \to L_v^\times/k_v^{\times}L_v^{\times 2}$.
	\item\label{step6} Return $V_{B,S}(Y) = V_0 \cap \prod_{v\in S}I_v$.
\end{enumerate}
\end{Algorithm}

\begin{Remark} We comment on various steps in the algorithm
\begin{enumerate}
\item[(\ref{step1})] This is a standard operation. For instance, in Magma this is accomplished with \texttt{pSelmerGroup}. The Norm map $L_v^\times/L_v^{\times 2}\to k_v^\times/k_v^{\times 2}$ can be computed by taking norms of elements representing the basis elements.
\item[(\ref{step2})] We compute Hilbert symbols between pairs of basis elements.
\item[(\ref{step3}), (\ref{step4})] These steps reduce to linear algebra over $\F_2$.
\item[(\ref{step5})] We use the algorithm of~\cite[12.6.7]{BPS}.
\end{enumerate}
\end{Remark}

\begin{Remark}\label{R:YkEmpty} If $S$ contains all real or dyadic primes, all primes of bad reduction for the model $g(x,y,z) = 0$, and all primes below primes of bad reduction for the model $ux+vy+wz =0$ of the generic bitangent, then by \cite[Lemma 12.13]{BPS} we have that the image of $C_f \colon \Pic(X) \to L^\times/k^{\times}L^{\times 2}$ lies in the unramified outside $S$ subgroup. Thus, Theorem~\ref{newthm1} applies and Algorithm~\ref{alg1} may allow one to compute that $Y(k)'$ is empty.
\end{Remark}

\begin{Remark}\label{rmk:unram} In practice one can often find a smaller set $S$ than that mentioned in Remark~\ref{R:YkEmpty}. Moreover, if the subgroup $B$ is unramified at some prime $v \in S$, then it may be possible to avoid computing the full local image at $v$ in Step~\ref{step5} of the algorithm.
\begin{enumerate}[wide]
	\item The linear form $ux+vy+wz \in \calO_L[x,y,z]$ with divisor $2\beta \times \Delta$ (the generic bitangent) is only determined up to scaling by an element of $L^\times$. The map $C_f$ in Lemma~\ref{lem:descentmap} depends on the choice of scalar, but the restriction of $C_f$ to $\Pic^0(X_K)$ does not. By factoring the $\calO_L$-ideal $\langle u,v,w \rangle$ and finding generators for principal ideals supported on its prime factors one may be able to scale $ux+vy+wz$ so that the ideal generated by its coefficients contains fewer or smaller primes.
	\item If $v$ is an odd prime such that the Tamagawa number of the Jacobian $J = \Pic^0_X$ is odd, then the image of $C_f \colon \Pic^0(X_{k_v}) \to L_v^\times/k_v^\times L_v^{\times 2}$ is unramified. This follows from \cite[Lemma 7.1 and Lemma 10.5(a)]{BPS}. For such primes it follows that for any $i \ge 0$, the image of $\Pic^i(X_{k_v})$ is contained in a coset of the unramified subgroup of $L_v^{\times}/k_v^\times L_v^{\times 2}$. In particular, one can check if the image is unramified by computing the image of a single divisor class.
	\item Suppose $v$ is an odd prime such that the Tamagawa number is odd and the subgroup $B \subset (L^\times/L^{\times 2})_S$ is unramified at $v$. In this case it is not necessary to compute the full image of $C_f \colon Y(k_v)' \to L_v^\times/k_v^{\times}L_v^{\times 2}$. Indeed, $B$ is orthogonal to the unramified subgroup and $C_f(Y(k_v)')$ is contained in a coset of this. So the pairings $\langle C_f(x_v), \res_v(b) \rangle_v$ in~\eqref{Hilb2} do not depend on $x_v \in Y(k_v)'$. This means that $V_{B,S}(Y)$ is either empty or its image in $L_v^\times/k_v^\times L_v^{\times 2}$ is the entire image $C_f(Y(k_v)')$. Knowing the image of a single point in $Y(k_v)'$ will allow us to distinguish these cases. Particularly for large primes where computing the local image may be expensive, this can dramatically improve efficiency of the algorithm.
	\end{enumerate}
\end{Remark}

\subsection{Generalization to arbitrary curves}\label{sec:General}
Algorithm~\ref{alg1} above can be generalized to any smooth, projective and geometrically irreducible curve $X/k$ of genus at least $2$ as follows.\\

\noindent{\bf Hyperelliptic Curves:} If $X$ is hyperelliptic of genus $g$ and has points in all completions of $k$, then there is a degree $2$ map $X \to \PP^1$, which we may assume is not ramified at $\infty \in \PP^1$. Then $X$ has an affine model of the form $y^2 = h(x)$ with $h(x) \in \calO_k[x]$ square-free of even degree. Let $\Delta = \Spec(L)$ denote the $k$-scheme of ramification points of the map $X \to \PP^1$, so $L = k[x]/(h(x))$ and let $f = (x-\theta) \in k(X\times \Delta)$ where $\theta$ denotes the image of $x$ in $L$. In this case $[L:k]=2g+2$.\\

\noindent{\bf Nonhyperelliptic Curves:} Let $\Delta = \Spec(L)$ denote the $k$-scheme of odd theta characteristics on $X$. These are classes in $\Pic(\Xbar)$ of divisors $\beta_\delta \in \Div(\Xbar)$ such that $2\beta_\delta$ lies in the canonical class. Thus, there is a function $f \in k(X \times \Delta)^\times$ with divisor $\divv(f) = 2\beta - D \times \Delta$ for some $D \in \Div(X)$ representing the canonical class. In this case $[L:k]=2^{(g-1)}(2^g-1)$.\\

In both cases the data $(2,\Delta,[\beta])$ defines a `fake descent setup' for $X$ (see \cite[Definition 6.7]{BPS}). Lemma~\ref{lem:descentmap} goes through as stated using the function $f \in k(X\times \Delta)$. It is not difficult to show in the general situation that there is an explicitly computable finite set $S$ of primes of $k$ for which the local images at primes not in $S$ are all unramified. Thus, with the obvious modifications, one obtains an algorithm that takes as input a subgroup $B \subset \left(L^\times/ L^{\times 2}\right)_{S,N = 1}$ and computes a (possibly empty) subset $V_{B,S}(Y) \subset \prod_v L_v^\times/k_v^\times L_v^{\times 2}$ which contains the image of $Y(k)'$. In the hyperelliptic case, this is essentially \cite[Algorithm 5.1]{CreutzSrivastava}. The proof of correctness given there involves the Brauer group of $X$, which we have thus far avoided here. In Section~\ref{sec:Descent} we relate $V_{B,S}(Y)$ from Definition~\ref{def:VBS} to the descent and Brauer-Manin obstructions. 

\section{Examples of obstructions to rational points on quartic curves}\label{sec:examples}

Our implementation of the algorithm and code verifying these examples can be found at~\cite{Github}.

\subsection{An example from \cite{BPS} revisited}
The following was considered in \cite[Prop. 12.20]{BPS} where it was proved that $X(\Q)=\emptyset$ conditionally on GRH. As they note, removing the conditional part following their approach would require verifying the class group computation of the integer ring of a degree $28$ number field with Minkowski bound exceeding $10^{22}$. By way of contrast the approach described in this paper gives the same result unconditionally.
\begin{Example}
	Let $X$ be the curve in $\PP^2_\Q$ defined by \[x^4 + y^4 + x^2yz + 2xyz^2 - y^2z^2 + z^4 = 0.\] Then for all $v \le \infty$ we have $X(\Q_v) \ne \emptyset$, but $X(\Q) = \emptyset$.
\end{Example}

\begin{proof}[Sketch of proof]
	One easily checks that $X(\R) \ne \emptyset$ and that $X(\Q_v) \ne \emptyset$ for $v < 37$ using Hensel's lemma. The discriminant of this quartic form is $-2^8\cdot 5^2\cdot 1361\cdot 97103$. In particular, the discriminant has valuation at most $1$ for all primes larger than $37$. So, it follows from~\cite[Lemma 12.14]{BPS} that $X(\Q_v) \ne \emptyset$ for all $v \ge 37$. Hence $X$ is locally soluble.
	
	The bitangent algebra $L$ of $X$ is a degree $28$ number field, with ring of integers $\calO_L$ of discriminant  $2^{30}\cdot 5^{10}\cdot 1361^6\cdot 97103^6$. We find a linear form $ux + vy + wz \in \calO_L[x,y,z]$ defining the generic bitangent, such that the $\calO_L$-ideal $\langle u,v,w\rangle$ is supported on primes above primes in the set $S_0 = \{ \infty, 2 \}$. Therefore, by Remark~\ref{R:YkEmpty} the set $S = \{ \infty, 2,5,1361,97103\}$ satisfies the conditions in Theorem~\ref{newthm1}. To prove $X(\Q) = \emptyset$, we compute the set $V_{B,S}(X)$ where $B$ is a subset of $(\calO_{L,S_0}^\times/\calO_{L,S_0}^{\times 2})_{N = 1} \subset (L^\times/L^{\times 2})_{S,N=1}$. We obtain generators for $B$ from the relation matrix obtained when computing the conditional class group of $\calO_L$ as described in Section~\ref{sec:findingelements}. We expect that $B$ generates $(\calO_{L,S_0}^\times/\calO_{L,S_0}^{\times 2})_{N = 1}$, but this is only verified assuming GRH.
	
	Since the discriminant has valuation $1$ at the primes $1361$ and $97103$, the Tamagawa numbers here are odd by~\cite[Remark 12.11]{BPS}. The Tamagawa number at $p = 5$ is also odd, as can be verified by computing the component group of a regular model of $X$ using Magma's intrinsic \texttt{RegularModel}. Following Remark~\ref{rmk:unram}(2) we check that $C_f(X(\Q_v))$ is contained in the unramified subgroup. It follows that for any $(\ell_v) \in \prod_{v \in S} C_f(X(\Q_v))$ and any $b \in B$ we have $\langle (\ell_v), b \rangle_S = \langle \ell_2 , b \rangle_2 + \langle \ell_\infty, b\rangle_\infty$.
	
	We compute the images of $X(\R)$ and $X(\Q_2)$ under $C_f$ using the algorithm of~\cite[12.6.7]{BPS}. The map $C_f$ is constant on $X(\R)$ since $X(\R)$ consists of a single connected component. We find $4$ values in the image of $C_f\colon X(\Q_2) \to L_2^\times/\Q_2^\times L_2^{\times 2}$. For each of the four elements $(\ell_2,\ell_\infty) \in C_f(X(\Q_2))\times C_f(X(\R))$ we find some $b \in B$ such that $\langle \ell_2,b \rangle_2 + \langle \ell_\infty, b \rangle_\infty = 1/2$. It follows that $V_{B,S}(X) = \emptyset$ and so $X(\Q) = \emptyset$ by Theorem~\ref{newthm1}. Full computational details are too voluminous to reproduce here, but code to generate the required data and verify the claims made here is available in \cite{Github}*{\texttt{BMQExamples.m}}.
\end{proof}

\begin{Remark}\label{rem:notall}
	Assuming GRH, the subgroup $B$ used in the preceding example is the subgroup
	\[
		B = (\calO_{L,S_0}^\times/\calO_{L,S_0}^{\times 2})_{S_0,N = 1} \subset (\calO_{L,S}^\times/\calO_{L,S}^{\times 2})_{S,N = 1} \subset (L^\times/L^{\times 2})_{S,N=1}\,.
	\]
	There are $18$ primes of $\calO_L$ above the primes in $S \setminus S_0 = \{ 5,1361,97103 \}$, so the first containment above has index at least $2^{18-3} = 2^{15}$. Moreover, $\operatorname{Cl}(\calO_{L,S_0})[2] = \operatorname{Cl}(\calO_{L,S})[2] \simeq \Z/2\Z$. So, by~\eqref{eq:ClassGroup}, the set $B$ has index $2$ in $(L^\times/L^{\times 2})_{S_0,N=1}$ and index at least $2^{16}$ in $(L^\times/L^{\times 2})_{S,N=1}$. In particular, we did not need all of $(L^\times/L^{\times 2})_{S,N=1}$ to find an obstruction. In fact, since the local image $C_f (X(\Q_2)) \times C_f(X(\R))$ has size $4$, it is clear in retrospect that the minimal subsets of $B_0 \subset (L^\times/L^{\times 2})_{S_0,N=1}$ for which $V_{B_0,S}(X) = \emptyset$ have size at most $4$.
\end{Remark}

\subsection{A database of plane quartics of small discriminant}
The database described in \cite{QL} contains $82241$ smooth plane quartic curves over $\Q$ with discriminant less than $10^7$. Among these there are $148$ curves that are everywhere locally solvable, but contain no $\Q$-rational points of small height. Using the algorithm described in this paper we found an obstruction to the existence of rational points for over 90\% of these, yielding $135$ plane quartic curves that are counterexamples to the Hasse principle. This leaves only $13$ curves in the database where the techniques described in this paper are not able to decide on existence of rational points. Equations for these curves can be found in~\cite[BMQexamples.m]{Github}.

We initially computed $V_{B,S}(X)$ for each of the $148$ curves, taking $S$ to be the minimal set of primes satisfying the conditions of Theorem~\ref{newthm1} and $S_0\subset S$ the primes above $2$ and those where the Tamagawa number is even. We computed the set $B\subset (L^\times/L^{\times 2})_{N=1,S_0}$, with equality conditional on GRH. This computation took approximately 24 hours using eight 2.6GHz CPUs (67 hours cpu time - parallelization was used for generating the relation matrices, but the rest was carried out using a single core) and found obstructions for $131$ of the curves. For the $17$ curves remaining we ran the algorithm again now taking $S$ and $S_0$ to include additional primes up to $5000$. This found four additional curves with an obstruction. 

\subsubsection{Using good primes}
The largest `good prime' which played a decisive role when rerunning the algorithm on these $17$ curves was $v = 11$, for the curve
\[
	X \colon x^4 + xy^3 + 4y^4 + x^3z + 2x^2z^2 - 2xyz^2 - 4y^2z^2 + xz^3 + 
z^4 = 0
\]
of discriminant $6037072 =  2^4\cdot 127 \cdot 2971$. For this curve, the Tamagawa numbers are all odd and we found a function $f$ such that the local images of $C_f$ are unramified at all odd primes. With $S$ the set of primes above $\{ \infty, 2, 127, 2971 \}$, $S_0$ the set of primes above $\{2,\infty\}$, and $B$ a subgroup of $(L^\times/L^{\times 2})_{N=1,S_0}$ equal to the whole group conditionally on GRH, we found $V_{B,S}(X) \ne \emptyset$. However, when including the primes above $11$ in $S$ and $S_0$ we found $V_{B,S}(X) = \emptyset$. One can give a geometric interpretation of this phenomenon using the torsors of type $\lambda_\Delta$ defined in Section~\ref{sec:setups} below. For this curve, there is a torsor $Z \to X$ of type $\lambda_\Delta$ which is locally soluble at all primes other than $v = 11$.

\subsubsection{Testing}
We have also run our code on thousands of curves in the database that do have obvious rational points and checked that the images of these rational points under $X(\Q) \stackrel{C_f}\longrightarrow (L^\times/k^\times L^{\times 2})_S \stackrel{\res_S}\longrightarrow \prod_{v \in S} L_v^\times/k_v^\times L_v^{\times 2}$ land in the computed set $V_{B,S}(X)$. 

\section{Computing the index of a plane quartic}\label{sec:Index}

The \defi{index} of a variety $X$ over a field $K$ is the GCD of the degrees of the closed points on $X$. If $X/K$ is a smooth geometrically integral plane quartic curve the index divides $4$. It is equal to $1$ if and only if $\Pic^1(X) \ne \emptyset$ and equal to $4$ if and only if $\Pic^2(X) = \emptyset$. If one has $\Pic^1(X) = \emptyset$, then the index may be either $2$ or $4$. 

\begin{Lemma}\label{L:curve_index}
	Let $X/K$ be a smooth plane quartic curve of index $I(X)$ over a field $K$.
	\begin{enumerate}
		\item If $I(X) = 1$, then there is a closed point of degree $1$ or $3$ on $X$.
		\item If $I(X) = 2$, then there is a closed point of degree $2$ or $6$ on $X$.
	\end{enumerate}
\end{Lemma}

\begin{proof}
	If $I(X) = 1$, then $X$ contains a $K$-rational divisor $D$ of degree $1$. Then $3D$ has degree $3$ and is linearly equivalent to an effective divisor by Riemann-Roch. This divisor is either a closed point of degree $3$ or contains a closed point of degree $1$ in its support. If $I(X) = 2$, then $X$ contains a $K$-rational divisor $D$ of degree $2$. Then $D + \kappa_X$ has degree $6$ and is linearly equivalent to an effective divisor, again by Riemann-Roch. Either this is a closed point of degree $6$ or contains a closed point of degree $2$ in its support. Note that it cannot contain any divisors of odd degree in its support because the index is $2$.
\end{proof}

For a local field $K$, there are only finitely many extensions of $K$ of any given degree. Using Hensel's lemma one can check for rational points over these extensions. This gives a practical algorithm to compute the index of a plane quartic curve over a local field. Given a plane quartic curve over a global field, one can compute the indices of $X$ over all of the completions (for primes of good reduction the index is 1 because the curve has effective divisors of any sufficiently high degree, thanks to the Weil bounds). The least common multiple of the local indices gives a lower bound for the global index. For a plane quartic curve, the least common multiple of the local indices is equal to the maximum local index, since the indices divide $4$ which is a prime power.

Finding that $V_{B,S}(\Pic^i_X)=\emptyset$  allows us in many cases to obtain a lower bound for the index that is larger by a factor of $2$.

\begin{Proposition}\label{prop:index}
	Let $X/k$ be a smooth quartic curve over a number field $k$ and $B,S$ as in Algorithm~\ref{alg1} satisfying the conditions of Theorem~\ref{newthm1}. 
	\begin{enumerate}
		\item If the maximum local index of $X$ is $1$ and $V_{B,S}(\Pic^1_X) = \emptyset$, then $I(X) \in \{2,4\}$. 
		\item If the maximum local index of $X$ is $2$ and $V_{B,S}(\Pic^2_X) = \emptyset$, then $I(X) = 4$.
	\end{enumerate}
\end{Proposition}

\begin{Remark}\label{rmk:pic12}
	If the maximum local index is $1$, then $V_{B,S}(\Pic^2_X) \ne \emptyset$. The reason for this is as follows. If $D_v \in \Pic^1(X_{k_v})$, then $C_f(2D_v) = 1$ and since $\Pic^2(X_{k_v}) = 2D_v + \Pic^0(X_{k_v})$, it follows that the image of $\Pic^2(X_{k_v})$ under $C_f$ contains $1$ for all $v$. But then $1 \in V_{B,S}(\Pic^2_X) \ne \emptyset$.
\end{Remark}

The next lemma is helpful in computing the local images $C_f(\Pic^2(X_{k_v}))$ at primes $v$ where the index is $2$.

\begin{Lemma}\label{lem:Lichtenbaum}
	Let $X/K$ be a smooth quartic curve of index $2$ over a $p$-adic field $K$. Then the image of $\Pic^0(X)$ in $\Pic_X^0(K)/2\Pic_X^0(K)$ is a subgroup of index $2$. Moreover, 
	\[ \#C_f(\Pic^2(X)) \le \begin{cases} 2^{-1}\cdot\#J(K)[2] & \text{if $v_K(2)=0\,;$}\\ 2^2\cdot\#J(K)[2] & \text{otherwise.}\end{cases}\]
\end{Lemma}

\begin{proof}
	Let $X$ be a smooth quartic curve of index $I$. To ease notation let $J = \Pic_X^0$ and $J^1 = \Pic^1_X$, which is a torsor under $J$. The period $P$ of $X$ is the order of $J^1$ in the group $\HH^1(K,J)$ and this integer divides $I$. Since $K$ is a $p$-adic field, \cite{Lichtenbaum} gives that $I = P$, except possibly when $(g-1)/P$ is odd. In our case the genus $g$ is $3$, so we cannot have $P = 1$ and $I = 2$. Hence $I = P = 2$.
	
	The Tate pairing $J(K)/2J(K) \times \HH^1(K,J)[2] \to \Br(K)[2]$ is nondegenerate. Since $J^1$ is a nontrivial element in $\HH^1(K,J)[2]$, the map $\Theta\colon J(K)/2J(K) \to \Br(K)[2]$ given by pairing with the class of $J^1$ is nonzero. By \cite[Proof of Corollary 1]{Lichtenbaum} this map sits in an exact sequence
	\[
		0 \longrightarrow \frac{\Pic^0(X)}{\Pic^0(X)\cap 2J(K)} \longrightarrow \frac{J(K)}{2J(K)} \stackrel{\Theta}\longrightarrow \Br(K)[2]\,.
	\]
	We conclude that the image of $\Pic^0(X)$ in $J(K)/2J(K)$ has index $2$, since $\Br(K)[2]$ has order $2$.
	
	For the second statement,~\cite[Corollary 6.15]{BPS} shows that the map $C_f$ defined on $\Pic^0(X)$ factors through $\Pic^0(X)/\Pic^0(X)\cap 2J(K)$. So $\#C_f(\Pic^2(X)) = \#C_f(\Pic^0(X)) \le 2^{-1}\#J(K)/2J(K)$. The size of $J(K)/2J(K)$ given in \cite[Remark 11.6]{BPS}, which leads to the bound in the statement.
\end{proof}

\subsection{Examples}

\subsubsection{Local index $4$}
In the database of plane quartic curves over $\Q$ described in~\cite{QL} there is just one curve with maximum local index $4$ (hence also global index $4$). This is the curve of discriminant $2201024$ given by the vanishing of
\[x^4 + x^3y + x^2y^2 + xy^3 + y^4 + x^3z + xy^2z - y^3z + 3x^2z^2 + 
    y^2z^2 + xz^3 - yz^3 + z^4\,.\]
There is a unique prime ($v = 2$) where the local index is $4$, hence this curve is `odd' in the sense of \cite{PoonenStoll}, meaning that $\Pic^2_X$ represents a nontrivial element in the Tate-Shafarevich group of its Jacobian and this group has odd $2$-rank. 

\subsubsection{Local index $2$}
There are 479 curves in the database with maximum local index $2$, of which all but $15$ have obvious quadratic points (hence have global index $2$). To find `obvious' quadratic points we searched for points on $X$ of small height over fields $\Q(\sqrt{d})$ with $|d| < 200$. In contrast to the situation for $\Q$-points where the lack of obvious points gives high confidence that there are in fact no points, it is plausible that some of these $15$ curves have quadratic points over larger quadratic fields. In the example below we use the methods in this paper on one of those $15$ curves to prove that the global index is in fact $4$, despite there being no local obstruction to having index $2$.

\begin{Proposition}
	Let $X \subset \PP^2_\Q$ be the curve given by the vanishing of
	\begin{multline*}
    x^4 + x^3y + 4x^2y^2 + 2xy^3 + 3y^4 + 2x^3z + 2x^2yz + 2xy^2z \\
    + 3y^3z + x^2z^2 + 7xyz^2 + 3y^2z^2 - 4xz^3 - 6yz^3 + 2z^4
\end{multline*}
	Then $X$ is a smooth genus $3$ curve of index $4$, but for every prime $v \le \infty$ the index of $X_{\Q_v}$ is at most $2$.
\end{Proposition}

\begin{proof}
	It is straightforward to verify that $X(\Q_v) \ne \emptyset$ for all $v \not\in \{ 2 , \infty \}$, and one finds points over quadratic extensions of $\Q_v$ for $v \in \{2, \infty\}$. Hence there is no local obstruction to having index $2$. We will use Algorithm~\ref{alg1} to show that $V_{B,S}(\Pic^2_X) = \emptyset$, where $S = \{2,\infty\}$ and $B$ is a subset of $(L^\times/L^{\times 2})_{S,N=1}$ which generates $\calO_{L,S}^\times/\calO_{L,S}^{\times 2}$ conditionally on GRH.  Proposition~\ref{prop:index} then applies to show that $X$ has index $4$.
		
	The bitangent algebra splits as a product $L = \Q(\sqrt{-5})\times K_4\times K_6 \times K_{16}$, where $K_i$ is a field of degree $i$ over $\Q$. The discriminant of $X$ is $7331200 = 2^7\cdot 5^2 \cdot 29 \cdot 79$ and we find a linear form defining the generic bitangent such that the ideal generated by its coefficients is supported at the primes dividing the discriminant.
	
	For the primes $v \in \{ 5,29,79\}$, $\Pic^1(X_{\Q_v}) \ne \emptyset$, so the local images $C_f(\Pic^2(X_{\Q_v}))$ must contain $1$ by Remark~\ref{rmk:pic12}. The Tamagawa numbers at these primes are odd, so the images of $\Pic^2(X_{\Q_v})$ are contained in some coset of the unramified subgroup by Remark~\ref{rmk:unram}. Since these cosets contain $1$, these images must be unramified. Since $X(\R) = \emptyset$, we have that $\Pic^{\textup{even}}(X_\R)$ is generated by the classes of divisors that are pairs of conjugate $\C$-points. Hence the image of $C_f \colon \Pic^2(X_\R) \to (L\otimes \R)^\times/\R^{\times}(L\otimes \R)^{\times 2}$ is trivial. In the algorithm, we can take $B \subset (L^\times/L^{\times 2})_{\{2,\infty\},N=1}$ and avoid computing local images at the primes other than $v = 2$.
		
	Let us now describe the computation of the local image $C_f(\Pic^2(X_{\Q_2}))$. We compute the images of random elements of $\Pic^2(X_{\Q_2})$ under the map $C_f$, easily finding a set of images with an affine span of size $2^4$, which we expect to be the full image. To verify this we compute an upper bound for the size of the image using Lemma~\ref{lem:Lichtenbaum}. First, we check that $X$ has no points of degree $1$ or $3$ over $\Q_2$ as discussed in the paragraph after Lemma~\ref{L:curve_index}. Then by Lemma~\ref{L:curve_index} the index of $X_{\Q_2}$ is $2$.  As outlined in \cite[Section~12.3]{BPS}, the action of Galois on the bitangents fully determines that Galois structure of $J[2]$. In this case we find that $J[2](\Q_2)$ has order $2^2$. So, by Lemma~\ref{lem:Lichtenbaum}, the size of $C_f(\Pic^2(X_{\Q_2}))$ is at most $2^4$. Since this agrees with the space spanned by the images already found, we have determined $C_f(\Pic^2(X_{\Q_2}))$.
	
	Since the local images are unramified at odd primes and trivial at $\R$, the pairing between $B= (L^\times/L^{\times 2})_{\{2,\infty\},N=1}$ and $\prod_{v \in \{2,\infty\}} C_f(\Pic^2(X_{\Q_v}))$ reduces to the pairing of $B$ with $C_f(\Pic^2(X_{\Q_2}))$. An explicit computation with Hilbert symbols checks that every element of $C_f(\Pic^2(X_{\Q_2}))$ pairs nontrivially with some element of $B$ under $\langle \,,\rangle_2$. We thus conclude that $V_{B,S}(\Pic_X^2)$ is empty and so $\Pic^2(X) = \emptyset$ by Theorem~\ref{newthm1}. Hence the index of $X$ is $4$ by Proposition~\ref{prop:index}.
\end{proof}

\begin{Remark}
	In the proof above, Lemma~\ref{lem:Lichtenbaum} allows us to obtain a sharp upper bound for the size of the local image. In turn, this allows us to compute the local image by considering random points and the image they span. Without a sharp upper bound, one would have to compute the image systematically, which may be impractical.
\end{Remark}

\subsubsection{Local index $1$}
The remaining 81751 curves in the database have maximum local index 1. All but 117 of these have obvious rational or cubic points (and hence have global index $1$). To find `obvious' cubic points we searched for points of small naive height over cubic fields of absolute discriminant at most 4000. Among these 117, all but one have an obvious quadratic or sextic point (and hence have global index in the set $\{1,2\}$). The following is an example of one of these where we can prove that the global index is $2$, despite there being no local obstruction to having index $1$.

\begin{Proposition}\label{prop:Pic1}
	Let $X \subset \PP^2_\Q$ be the curve given by the vanishing of
	\begin{multline*}
		x^4 + 3x^2y^2 - xy^3 + 3y^4 + 3x^3z + 3x^2yz + 4xy^2z + 6y^3z \\ - x^2z^2 + 5xyz^2 - 2y^2z^2 - 4xz^3 - 5yz^3 + 2z^4\,.
	\end{multline*}
	Then $X_{\Q_v}$ has index $1$ for all primes $v$, but $X$ has index $2$.
\end{Proposition}

\begin{proof}
	This curve has $\Q_v$-points for all $v \le \infty$, so the local indices are all $1$. It has the degree $2$ point $( -\sqrt{3}-1 : 0 : 1)$, so the index is at most $2$. 
	
	The bitangent algebra of $X$ is a degree $28$ number field $L$. The discriminant of the defining quartic is $2560173 = 569\cdot 4517$. We find a linear form $ux + vy + wz \in \calO_L[x,y,z]$ defining the generic bitangent such that the ideal $\langle u,v,w\rangle$ is supported on primes above $2953$. Hence the set $S = \{ \infty,2,569,2953,4517\}$ satisfies the conditions of Theorem~\ref{newthm1}. Moreover, the local images $C_f(\Pic^1(X_{\Q_v}))$ are contained in a coset of the unramified subgroup for $v \ne 2,\infty$ by Remark~\ref{rmk:unram}(2). We compute a subgroup $B \subset (L^\times/L^{\times 2})_{\{2,\infty\},N = 1}$, equal to the whole group conditionally on GRH. 
	
	The image of $\Pic^1(X_\R)$ is equal to the image of $X(\R)$, which consists of two connected components. Since $C_f$ is locally constant, we compute the image of $\Pic^1(X_\R)$ by computing $C_f(x)$ for one $x \in X(\R)$ on each component. For $v = 2$, the Galois action on the bitangents shows that $J[2](\Q_2) = 0$, and so by \cite[Remark 11.6]{BPS} the size of $C_f(\Pic^1(\Q_2))$ is bounded above by $2^3$. We find degree $1$ divisors on $X_{\Q_2}$ supported on closed points of degree $1$ and $2$ whose images have an affine span of size $2^3$. Thus we have computed the full local image of $\Pic^1(X_{\Q_2})$. We can then use Remark~\ref{rmk:unram} to compute $V_{B,S}(\Pic^1_X)$ without having to compute the full local image at the other primes of $S$. We find that $V_{B,S}(\Pic^1_X) = \emptyset$ and so $X$ has index $2$ by Proposition~\ref{prop:index}.
\end{proof}

\begin{Remark}\label{rem:Sha}
	It follows from Corollary~\ref{cor:Sha} below that, for the curve $X$ in Proposition~\ref{prop:Pic1}, we have that $\Pic^1_X$ represents a nontrivial element in the Tate-Shafarevich group $\Sha(\Pic^0_X)[2]$ and that $\#\Sha(\Pic^0_X)[2] \ge 4$.
\end{Remark}

\begin{Remark}
	It may be possible to also compute $V_{B,S}(\Pic^1_X)$ for the other 116 curves with maximum local index $1$ and no obvious odd degree points, but we have not yet done so. The main obstacle is the computation of the local images $C_f(\Pic^1(X_{\Q_v}))$. As described in \cite{BPS} this can be done by computing the local images of $X(K)$ for all extensions $K/\Q_v$ of degree up to $3$. One can obtain an upper bound for the size of the image from the Galois action on $J(\Q_v)[2]$ and then compute the images of random points until they span a space of the correct size. However, this upper bound need not always be sharp, in which case one would have to compute the full image of $X(K)$ for all extensions $K/\Q_v$ of degree up to $3$. It is unclear whether this will be feasible in all cases. 
\end{Remark} 
%

\section{Brauer and Selmer sets}\label{sec:Descent}

Let $X/k$ be a smooth, projective and geometrically irreducible nonhyperelliptic curve of genus $\ge 3$ over a number field and $Y \subset \Pic^i_X$ a smooth closed subscheme such that $\Pic^0_Y = \Pic^0_X$. This holds for example if $Y = X$ or $Y = \Pic^i_X$. 
In this section we explain how the set $V_{B,S}(Y)$ computed in Algorithm~\ref{alg1} (and its generalization described in Section~\ref{sec:General}) is related to finite abelian descent and Brauer-Manin obstructions. This section is not needed to prove correctness of the algorithm or for the examples above (other than in Remark~\ref{rem:Sha}), but provides some context for the method as well as insight into which elements in $(L^\times/L^{\times 2})_{S,N=1}$ are (or are not) useful for computing the obstruction.

\subsection{Selmer sets}
Consider the set
\begin{align*}
	\Sel^f(Y) &:= \{ \ell \in L^\times/k^\times L^{\times 2} \::\; \res_v(\ell) \in C_f(Y(k_v)') \text{ for all $v$} \}\,
\end{align*}
from \cite[Definition 9.4]{BPS} and ~\cite[Definition 5.1]{CreutzGenJacs}. This is related to $V_{B,S}(Y)$ as follows.
\begin{Proposition}\label{prop:Sel}
	Let $S,B,V_{B,S}(Y)$ be as in Algorithm~\ref{alg1}.  Assume that the images of $C_f\colon Y(k_v)' \to L_v^\times/k_v^{\times}L_v^{\times 2}$ are unramified at all primes not in $S$. Then the image of $\Sel^f(Y)$ under the restriction map $L^\times/k^\times L^{\times 2} \to \prod_{v \in S}L_v^\times/k_v^\times L_v^{\times 2}$ is contained in $V_{B,S}(Y)$.
\end{Proposition}

\begin{proof}
	By assumption $\Sel^f(Y)$ is contained in $(L^\times/k^\times L^{\times 2})_S$. Hence, the image of $\Sel^f(Y)$ under the restriction map is orthogonal to $B$ with respect to $\langle\,,\,\rangle_S$ by Theorem~\ref{thm:pairing}. Then the image is contained in $V_{B,S}(Y)$ as claimed.
\end{proof}

\begin{Corollary}\label{cor:Sha}
	Suppose that $\Pic^1(X_{k_v}) \ne \emptyset$ for all $v$. If $V_{B,S}(\Pic^1_X) = \emptyset$, then $\Pic^1_X$ is not divisible by $2$ in $\Sha(\Pic^0_X)$ and $\#\Sha(\Pic^0_X)[2] \ge 4$.
\end{Corollary}

\begin{proof}
	The assumption that $\Pic^1(X_{k_v}) \ne \emptyset$ for all $v$ implies that the Cassels-Tate pairing on $\Sha(\Pic^0_X)$ is alternating \cite{PoonenStoll}, from which one deduces that $\Sha(\Pic^0_X)[2]/2\Sha(\Pic^0_X)[4]$ has square order. Under the further assumption that $\Sel^f(\Pic^1_X) = \emptyset$~\cite[Theorem 5.3]{CreutzGenJacs} shows that $\Pic^1_X$ represents a nontrivial in this group, hence its order is at least $4$.
\end{proof}
%
%

\begin{Remark}
	The size of $\Sel^f(\Pic_X^0)$ is often used to compute an upper bound for the rank of the Mordell-Weil group of the Jacobian $\Pic^0_X$. For large enough $S$ and $B$ we do have a surjective map $\Sel^f(\Pic_X^0)\twoheadrightarrow V_{B,S}(\Pic_X^0)$, but it seems unavoidable that determining the size of the kernel involves information on the class and unit groups of $L$. Consequently, we do not see how the computational advantages of our method can extend to the problem of bounding the rank of the Jacobian.
\end{Remark}

\subsection{Fake descent setups and torsors of type $\lambda_\Delta$}\label{sec:setups}

As described in Section~\ref{sec:General}, the data we consider for defining $V_{B,S}(Y)$ gives rise to a \emph{fake descent setup} $(n,\Delta,[\beta])$ in the sense of \cite{BPS}*{Definition~6.7}, with $n=2$. As in \cite[Section 6.4]{BPS} this gives rise to a Galois module $E = (\Z/n\Z)^\Delta_{\deg 0}$ and an exact sequence $0 \to R \to E \to J^\vee[n] \to 0$, where $J^\vee = \Pic^0_X = \Pic^0_Y$ and $R$ is the kernel. Dualizing gives an exact sequence $0 \to J[n] \to E^\vee \to R^\vee \to 0$, where $J = \Alb^0_X = \Alb^0_Y$ is the Albanese variety. Let $\lambda_\Delta\colon E \to J^\vee[n] \to \Pic(\Ybar)$ be the composition.

In the language of \cite[Chapter 2]{Skorobogatov}, $\lambda_\Delta$ is a \emph{type map} and by \cite[Corollary 2.3.9]{Skorobogatov} determines a (unique up to isomorphism) torsor $\overline{Z} \to \Ybar$ under $E^\vee$. A torsor $Z \to Y$ which becomes isomorphic to this after base change to $\kbar$ is called a \emph{torsor of type $\lambda_\Delta$}.

From the exact sequence above, the $E^\vee$-torsor $\overline{Z} \to \Ybar$ factors as $\overline{Z} \to \overline{W} \to \Ybar$, where $\overline{W} \to \Ybar$ is an $E^\vee/J[n] = R^\vee$-torsor whose \emph{type} is the restriction of $\lambda_\Delta$ to $R$. Since this restriction is the $0$ map, $\overline{W} = \Ybar \times R^\vee(\kbar)$ is the trivial $R^\vee$-torsor. The connected components of $\overline{Z}$ are the $J[n]$-torsors $\overline{Z}_r \to \Ybar \times \{r\}$, $r \in R^\vee(\kbar)$, all of which have the type map $J^\vee[n] \to \Pic(\Ybar)$. Such torsors are obtained by pulling back multiplication by $n$ on $\Alb^0_{\Ybar}$ and are often called \emph{$n$-coverings}. A torsor $Z \to Y$ of type $\lambda_\Delta$ is therefore a family of $n$-coverings of $Y$, parametrized by an $R^\vee$-torsor over $\Spec(k)$.

The torsors of type $\lambda_\Delta$ cut out a subset of the adelic points, $Y(\A_k)^{\lambda_\Delta} = \bigcup \pi(Z(\A_k))$, the union ranging over all torsors $\pi\colon Z \to Y$ of type $\lambda_\Delta$. Every rational point on $Y$ lifts to some torsor of type $\lambda_\Delta$, and so $Y(k) \subset Y(\A_k)^{\lambda_\Delta}$. It follows that if $Y(\A_k)^{\lambda_\Delta} = \emptyset$, then $Y(k) = \emptyset$. In this case, one says there is an obstruction to the existence of rational points on $Y$ coming from torsors of type $\lambda_\Delta$. This is a particular case of the \emph{finite abelian descent obstruction}. See~\cite[5.3]{Skorobogatov} for more details.

\begin{Remark}
If there is a $Y$-torsor of type $\lambda_\Delta$, then their isomorphism classes are parameterized by $\HH^1(k,E^\vee)$. From the Galois cohomology of $1\to\mu_n\to\mu_n^\Delta \to E^\vee\to 1$ and Hilbert's Theorem 90 we get an injective map $L^\times/k^\times L^{\times n} \to \HH^1(k,E^\vee)$. It can be shown that if $Y(k_v)' = Y(k_v)$ for all primes $v$, then all torsors with points everywhere locally lie in the image and, moreover, that $\Sel^f(X) \subset L^{\times}/k^\times L^{\times 2}$ parameterizes the set of torsors $\pi\colon Z\to X$ of type $\lambda_\Delta$ such that $Z(k_v)\neq \emptyset$ for all $v$. 
\end{Remark}

\subsection{Brauer sets}

The \'etale cohomology group $\Br(Y) := \HH^2(Y,\G_m)$ is called the \emph{Brauer group} of $Y$. By the purity theorem $\Br(Y)$ can be identified with the unramified subgroup of the Brauer group of the function field of $Y$. A class in $\Br(Y)$ can be evaluated at a point $x \in Y(K) = \Hom(\Spec(K),Y)$ defined over some extension $K/k$ by pulling back to yield an element of $\Br(K)$. For a local point $x_v \in Y(k_v)$ this can be composed with the invariant map $\inv_v \colon \Br(k_v) \to \Q/\Z$. For given $\alpha \in \Br(Y)$, the resulting maps are nonzero for only finitely many primes. In this way, each $\alpha \in \Br(Y)$ determines a map
\[
Y(\A_k) = \prod_{v}Y(k_v) \ni (x_v) \mapsto \sum_v \inv_v(\alpha(x_v)) \in \Q/\Z\,.
\]
For a subset $B' \subset \Br(Y)$, let $Y(\A_k)^{B'}$ denote the subset of adelic points which map to $0 \in \Q/\Z$ under all of the maps induced by the $b \in B'$. By global class field theory one has $Y(k) \subset Y(\A_k)^{B'}$.

The type map $\lambda_\Delta$ gives rise to a subgroup of the Brauer group of $Y$. From the Hochschild-Serre spectral sequence, one has a morphism $r\colon \Br_1(Y) \to \Br_1(Y)/\Br_0(Y) \simeq \HH^1(k,\Pic(\Ybar))$. One defines $\Br_{\lambda_\Delta}(Y) = r^{-1}({\lambda_\Delta}_*(\HH^1(k,E)))$. The descent theory of abelian torsors of finite type \cite[Theorem 6.1.2]{Skorobogatov} gives an equality of sets $Y(\A_k)^{\lambda_\Delta} = Y(\A_k)^{\Br_{\lambda_\Delta}(Y)}$ and hence equivalence of the obstructions.

The definition of $E$ gives an exact diagram of Galois modules (the bottom row follows from the snake lemma applied to the top two rows and defines $Q$):

\[\begin{tikzcd}
	& 0 & R & E & {J^\vee[n]} & 0 \\
	& 0 & {(\Z/n\Z)^\Delta} & {(\Z/n\Z)^\Delta} & 0 \\
	0 & {J^\vee[n]} & Q & {\Z/n\Z} & 0
	\arrow[from=1-2, to=1-3]
	\arrow[from=1-3, to=1-4]
	\arrow[hook, from=1-3, to=2-3]
	\arrow[from=1-4, to=1-5]
	\arrow[hook, from=1-4, to=2-4]
	\arrow[from=1-5, to=1-6]
	\arrow[from=2-2, to=2-3]
	\arrow[equal, from=2-3, to=2-4]
	\arrow[two heads, from=2-3, to=3-3]
	\arrow[from=2-4, to=2-5]
	\arrow[two heads, from=2-4, to=3-4]
	\arrow[from=3-1, to=3-2]
	\arrow[from=3-2, to=3-3]
	\arrow[from=3-3, to=3-4]
	\arrow[from=3-4, to=3-5]
\end{tikzcd}\]
Taking Galois cohomology we obtain an exact diagram
\begin{equation}\label{diagramQ}
\begin{tikzcd}
	\HH^0(k,Q) \arrow[r]\arrow[d]& \Z/n\Z \arrow[d]\arrow[r]& 0 \arrow[d] \\
	\HH^1(k,R) \arrow[r] & \HH^1(k,E) \arrow[r] \arrow[d,two heads]& \HH^1(k,J^\vee[n])\\
	&P
\end{tikzcd}
\end{equation}
where $P = \ker\left(\HH^1(k,\Z/n\Z^\Delta) \to \HH^1(k,\Z/n\Z)\right)$. The image of $\Br_{\lambda_\Delta}(Y)$ in $\Br(Y)/\Br_0(Y)$ is the image of the map $\HH^1(k,E) \to \HH^1(k,J^\vee[n]) \to \HH^1(k,\Pic_Y) \simeq \Br_1(Y)/\Br_0(Y)$. By exactness, this map $\HH^1(k,E) \to \Br(Y)/\Br_0(Y)$ factors through $P$, so there is a surjective map $\Gamma\colon P \to \Br_{\lambda_\Delta}(Y)/\Br_0(Y)$.

In the case $n = 2$, Hilbert's Theorem 90 yields $P \simeq (L^\times/L^{\times 2})_{N=1}$. So there is a surjective homomorphism $(L^\times/L^{\times 2})_{N=1} \to \Br_{\lambda_\Delta}(Y)/\Br_0(Y)$. The following proposition gives an explicit description of this map in the case $Y = X$, relating it to the pairing~\eqref{Hilb2} used to define $V_{B,S}(X)$.

\begin{Proposition}\label{prop:Br}
	There is a homomorphism 
	\[
	\gamma \colon \left(L^\times/L^{\times 2}\right)_{N=1} \to \Br_{\lambda_\Delta}(X)\,,
	\]
	whose image surjects onto $\Br_{\lambda_\Delta}(X)/\Br_0(X)$, defined by $$\gamma(\ell) = \Cor_{L/k}\left(\calA(f,\ell)\right)\,,$$ where $\Cor_{L/k}\colon \Br(k(X)\otimes L) \to \Br(k(X))$ is the corestriction map and $\calA(f,\ell)$ is the quaternion algebra over $k(X)\otimes L$ defined by $f,\ell \in (k(X)\otimes L)^\times$. If $S$ is a finite set of primes of $k$ such that
	\begin{enumerate}
		\item for all $v \not\in S$, the local images $C_f(X(k_v)) \subset L_v^\times/k_v^{\times}L_v^{\times 2}$ are unramified, and
		\item $\ell$ represents a class in $(L^\times/L^{\times 2})_{S,N=1}$,
	\end{enumerate}
	then for any adelic point $(x_v) \in X(\A_k) = \prod X(k_v)$ we have
	\[
	\sum_{v \in \Omega_k} \inv_v(\gamma(\ell)(x_v)) = \langle (C_f(x_v))_{v \in S}, \ell \rangle_S \in \Q/\Z\,.
	\]
\end{Proposition}

\begin{proof}
	This follows similarly to the proof of \cite[Theorem 1.1]{CV}.
\end{proof}

\begin{Corollary}\label{Cor:Br}
	Let $S$ be a finite set of primes, let $B \subset (L^\times/L^{\times 2})_{S,N=1}$ and consider $\gamma(B) \subset \Br_{\lambda_\Delta}(X)$. If $X$ and $S$ satisfy the conditions of Theorem~\ref{newthm1}, then $V_{B,S}(X)$ is the image of $X(\A_k)^{\gamma(B)}$ under the map $X(\A_k) \to \prod_{v\in S}C_f(X(k_v))$.
\end{Corollary}

\begin{Corollary}\label{Cor:Br2}
	There is an explicitly computable finite set of primes $S$ such that with $B = (L^\times/L^{\times 2})_{S,N=1}$, the following are equivalent
	\begin{enumerate}
		\item\label{a1} $V_{B,S}(X) = \emptyset$,
		\item\label{a3} $X(\A_k)^{\Br_{\lambda_\Delta}(X)} = \emptyset$,
		\item\label{a4} $X(\A_k)^{\lambda_\Delta} = \emptyset$,
		\item\label{a2} $\Sel^f(X) = \emptyset$.
	\end{enumerate}
\end{Corollary}

\begin{proof}
	 
	The equivalence of the obstructions given by $\gamma(B) = \Br_{\lambda_\Delta,S}(X)$ and $\Br_{\lambda_\Delta}(X)$ for sufficiently large $S$ is a special case of \cite[Theorem 3.1]{CreutzSrivastava}. This gives the equivalence of~\eqref{a1} and~\eqref{a3}. The equivalence of~\eqref{a3} and~\eqref{a4} follows from \cite[Theorem 6.1.2]{Skorobogatov}. The implication~\eqref{a1} $\Rightarrow$~\eqref{a2} follows from Proposition~\ref{prop:Sel}. It only remains to show that~\eqref{a2} implies one of the other statements.
	
	We can assume the local images are unramified outside $S$, so that $\Sel^f(X)$ is contained in the finite set $(L^\times/k^\times L^{\times_2})_{S}$. So, if $\Sel^f(X) = \emptyset$, then there must be a finite set of primes $T$ containing $S$ such that the set
	\[
		A = \{ \ell \in (L^\times/k^\times L^{\times 2})_{S} \::\; \res_v(\ell) \in C_f(X(k_v)) \text{ for all $v \in T$} \}
	\]
	is empty. Since the local images are unramified at $T \setminus S$, $A$ is equal to the set 
	\[
		A' = \{ \ell \in (L^\times/k^\times L^{\times 2})_{T} \::\; \res_v(\ell) \in C_f(X(k_v)) \text{ for all $v \in T$} \}\,.
	\]
	With $B' = (L^\times/L^{\times 2})_{T,N=1}$ it follows from Theorem~\ref{thm:pairing} that $V_{B',T}(X) = A' = \emptyset$. By Proposition~\ref{prop:Br} this implies that $X(\A_k)^{\Br_{\lambda_\Delta}} = \emptyset$, whence~\eqref{a2} $\Rightarrow$~\eqref{a3}.
\end{proof}
%
%

\subsection{Brauer classes from $\HH^1(k,R)$}

By exactness of \eqref{diagramQ}, elements in the image of the map $\HH^1(k,R) \to \HH^1(k,E)$ map to $0$ in $\Br_1(X)/\Br_0(X)$. This observation provides some insight into what kind of elements from $(L^\times/L^{\times 2})_{N=1}$ can provide non-trivial information on rational points.

We give an explicit example for a plane quartic $X$ with an \emph{even theta characteristic}. We first review some of the classical theory of theta characteristics on plane quartics (see \cite{Dolgachev:CAG}*{Chapter~6}). Let us write $\kappa_X\in\Pic(X)$ for the canonical class on $X$. We say $\theta\in\Pic(X)$ is a \emph{theta characteristic} if $2\theta=\kappa_X$.
The contact points of a bitangent to $X$ sum to a theta characteristic. These are called \emph{odd} theta characteristics. There are $28$ of those and the algebra $L$ is the corresponding affine coordinate ring for them.

The remaining $36$ theta characteristics of $X$ are \emph{even} theta characteristics. The even and odd theta characteristics are related via \emph{Aronhold sets}. An Aronhold set is a collection of $7$ odd theta characteristics $\{\theta_1,\ldots,\theta_7\}$ such that no triple $\{\theta_i,\theta_j,\theta_k\}$ of distinct members has $\theta_i-\theta_j+\theta_k$ equal to an odd theta characteristic.

We have that $\theta_1+\cdots+\theta_7=7\theta_0$ for some even theta characteristic $\theta_0$. Indeed the $288$ Aronhold sets are partitioned into $36$ sets of $8$, where each set $A_{\theta_0}$ consist of those that sum to $7\theta_0$. Any two Aronhold sets from $A_{\theta_0}$ have a unique odd theta characteristic in their intersection, reflecting that $\binom{8}{2}=28$.

If we have an even theta characteristic $\theta_0$ on $X$ and write $A=A_{\theta_0}$ for the corresponding octet, then we obtain a homomorphism
$(\Z/2\Z)^A \to (\Z/2\Z)^\Delta$, sending each Aronhold set to its sum. Since these sums all have the same image in $\Pic(X)$, we see that this map restricts to a homomorphism $(\Z/2\Z)^A_{\deg 0} \to R.$

Let $L_A$ be the octic \'etale algebra that is the affine coordinate ring of $A$. We have a morphism $L_A^\times \to \Sym^2(L_A)^\times$ by taking the product of evaluation at the pair. With the observation about pairwise intersections made above, we see that $\Sym^2(L_A)=L_A\oplus L,$
and hence we obtain a map
$L_A^\times \to L^\times.$
However, on the level of cohomology, this represents the derived morphism $\HH^1(k,(\Z/2\Z)^A_{\deg 0}) \to \HH^1(k,R).$
As noted above, elements in the image of $\HH^1(k,R) \to \HH^1(k,E)$ map to $0$ in $\Br_1(X)/\Br_0(X)$, so for the purpose of finding elements that are non-trivial with respect to the $\langle\,,\,\rangle_S$-pairing, any $S$-unit from $L_A$ is not useful.


\end{document}